\documentclass[11pt]{article}
\usepackage{amsmath,amssymb}

\newtheorem{propo}{{\bf Proposition}}[section]
\newtheorem{coro}[propo]{{\bf Corollary}}
\newtheorem{lemma}[propo]{{\bf Lemma}} \newtheorem{theor}[propo]{{\bf
Theorem}} \newtheorem{ex}{{\sc Example}}[section]
\newtheorem{defn}{\bf Definition}
\newenvironment{proof}{{\bf Proof.}}{$\Box$}

\begin{document}
\vspace*{1.0in}

\begin{center} Nilpotency, Solvability and Frattini Theory for Bicommutative, Assosymmetric and Novikov algebras

\end{center}
\bigskip

\centerline {David A. Towers} \centerline {Department of
Mathematics, Lancaster University} \centerline {Lancaster LA1 4YF,
England}
\bigskip

{\bf Abstract}

\medskip 

This paper starts by showing that, for algebras in a certain class, the concepts of weak nilpotency and nilpotency coincide. It goes on to describe some solvability and nilpotency properties of bicommutative algebras, of assosymmetric algebras and of Novikov algebras and to introduce a Frattini theory for all of them. A description is also given for semisimple bicommutative algebras over any field.

\medskip


\section{Introduction}
\begin{defn} A nonassociative algebra $A$ is called {\bf right (respectively, left) commutative} if $(ab)c=(ac)b$ (respectively, $a(bc)=b(ac)$) for all $a,b,c\in A$; it is {\bf bicommutative} if it is both right and left commutative.
\end{defn}

One-sided commutative algebras first appeared in a paper of Cayley in 1857 (\cite{cay}). The study of bicommutative algebras was initiated by Dzhumadil'daev and Tulenbaev in 2003 (\cite{dt1}), where, in particular, they proved that the derived algebra of such an algebra is commutative and associative, and hence that any simple algebra is a field. They were also studied by Burde, Dekimpe and Deschamps, who called them $LR$-algebras, in 2009 (\cite {bdd}). They arise in a number of settings, including the study of affine actions on nilpotent Lie groups. Low dimensional nilpotent bicommutative algebras have been classified algebraically and geometrically in \cite{akk,kpv1,kpv2}.

\begin{defn} A {\bf left-symmetric} algebra $A$ is a vector space with a bilinear product, denoted by juxtaposition, which satisfies
\[ x(yz)-(xy)z=y(xz)-(yx)z.
\] The {\bf associator}, $(x,y,z)=(xy)z-x(yz)$ of any three elements $x,y,z$ in an algebra $A$ can be thought of as measuring the degree of associativity in $A$. Then the defining identity of an  left symmetric algebra can be written as 
\[
(x,y,z)=(y,x,z).
\]  Similarly, $A$ is {\bf right symmetric} if 
\[ (x,y,z)=(x,z,y).
\] We will say that $A$ is {\bf bisymmetric} if it is both left and right symmetric.
\end{defn}

Left symmetric algebras arise in different areas of mathematics and physics (see, for example, \cite{burde}). They have been studied by several authors (see, for example, \cite{ckm}) but a structure theory for such algebras appears to be elusive. However, when further conditions which also arise naturally are added, a nice structure theory can be found.

\begin{defn} Following Kleinfeld in \cite{klein}, the algebra $A$ is called {\bf assosymmetric} if
\[ (x,y,z)=(\sigma(x),\sigma(y),\sigma(z)) \hbox{ for all } x,y,z,\in A,
\] where $\sigma$ is an arbitrary permutation of $x,y,z$.
\end{defn}

Then the following is easy to see, since $S_3$ is generated by $(12)$ and $(23)$..

\begin{propo} The algebra $A$ is assosymmetric if and only it is bisymmetric.
\end{propo}

Assosymmetric algebras are not power associative or flexible, so it was many years before they were investigated in any detail. However, in 1986, Kleinfeld published a new paper (\cite{klein2}) which initiated their study in greater detail.

\begin{defn} A nonassociative algebra $A$ is called a (left) {\bf Novikov algebra} if it is left symmetric and right commutative. There is also an opposite (right) version satisfying right symmetry and left commutativity. 
\end{defn}

These algebras were introduced independently by Gelfand and Dorfman in 1979 as an algebraic approach to the Hamiltonian operator in mechanics (\cite{gd}) and by Balinskii and Novikov in 1985 in relation to hydrodynamics (\cite{bn}). In \cite{bn}, a question was raised about the simple algebras in this class and this was answered  in the same year by Zelmanov who showed that, over a field of characteristic zero, the simple Novikov algebras are fields (\cite{zel}). Low-dimensional Novikov algebras were classified in \cite{bg,kkk}, and a version of Engel's Theorem for such algebras was given in \cite{dt}.
\par

Both of these classes of algebras have attracted the attention of many authors and have a rich theory. They are both Lie admissible algebras; that is, commutator multiplication on them gives a Lie algebra. In certain areas they have some similarities (such as the simple algebras being commutative associative algebras), but in others they differ significantly, as we shall see later. For a fuller bibliography, see {\cite{drensky} for bicommutative algebras and \cite{sz} for Novikov algebras. The purpose of this paper is to study some further solvability and nilpotency properties and to introduce a Frattini theory for them.
\par

In section 2 we consider some concepts of nilpotency in general nonassociative algebras. In particular, we show that, in a class of algebras $\mathfrak{X}$ which is factor algebra and subalgebra closed and, for every $A\in \mathfrak{X}$, $IJ$ is an ideal of $A$ whenever $I,J$ are ideals of $A$, a weaker version of nilpotency is equivalent to nilpotency. It is shown subsequently that all three classes of algebras considered here belong to $\mathfrak{X}$. 
\par

Section 3 is devoted to bicommutative algebras. We show that the Frattini ideal of such an algebra is nilpotent, that minimal ideals are zero algebras or are simple, and produce a decomposition of semisimple such algebras as a direct sum of fields extended by a zero subalgebra (one whose square is zero). It is also shown that, unlike Novikov algebras, the sum of two zero subalgebras need not be solvable; in fact, it can be semisimple. 
\par

In section 4, assosymmetric algebras $A$ are considered. It is shown that, over a field of characteristic different from $2, 3$, $\phi(A)$ is nilpotent. In section 5, Novikov algebras $A$ are introduced. It is shown that, if $A$ is solvable, then $\phi(A)$ is nilpotent. Also, that if $R$ is the solvable radical, then $AR$ is a nilpotent ideal of $A$. 
\par

In the final section, algebras with trivial Frattini ideal are considered. It is shown that, if an algebra $A$ has $\phi(A)$ nilpotent, then $A$ is $\phi$-free if and only if it splits over the sum of its zero ideals. Decomposition results are found for two of these classes of algebras, over a general field for bicommutative algebras, and over a field of characteristic zero for Novikove algebras with nilpotent radical. A consequence of the latter result is that, over a field of characteristic zero, the Frattini ideal of a Novikov agebra is nilpotent. The section finishes with two results concerning maximal subalgebras.
\par

Throughout, $A$ will denote a finite-dimensional nonassociative algebra over a field $F$. Algebra direct sum will be denoted by $\oplus$, whereas direct sum of the vector space structure alone will be denoted by $\dot{+}$. We will use $\subseteq$ to indicate inclusion, and $\subset$ for strict inclusion.

\section{Some General Results}
\begin{defn} If $B$ is a subalgebra of an algebra $A$, the {\bf idealiser} of $B$ in $A$ is $I_A(B)=\{a\in A \mid aB+Ba\subseteq B\}$. The {\bf annihilator} of $I$ in $A$ is $Ann_A(I)=\{a\in A \mid aI=Ia=0\}$. We will write $Ann(A)$ for $Ann_A(A)$.
\end{defn}

\begin{defn} For any algebra $A$, the {\bf derived series} of subalgebras $A^{(0)}\supseteq A^{(1)}\supseteq A^{(2)} \supseteq \ldots$ of $A$
is obtained by defining $A^{(0)}=A$, $A^{(i+1)}=(A^{(i)})^2$ for all $i\ge 0$. $A$ is called {\bf solvable} if $A^{(n)}=0$ for some $n$. Every algebra $A$ has a unique solvable ideal which we will call the {\bf solvable radical} of $A$ and denote by $R(A)$. We will call $A$ {\bf semisimple} if $R(A)=0$. (Note that this is different from the definition used by some authors.)
\end{defn}

\begin{defn} Put $A^1=$ $^1A=A^{[1]}=A$, $A^{n+1}=A^nA$, $^{n+1}A=A(^nA)$, $A^{[n+1]}=\sum_{i+j=n+1}A^{[i]}A^{[j]}$. We call $A$ {\bf right nilpotent} if $A^n=0$, {\bf left nilpotent} if $^nA=0$, {\bf weakly nilpotent} if it is both right and left nilpotent, and {\bf nilpotent} if $A^{[n]}=0$ for some $n\ge 1$.
\end{defn}

Clearly, an algebra is nilpotent if and only if there is an $n$ such that every product of $n$ elements is zero. It follows that $A^{[n]}$ is an ideal of $A$ for every $n\ge 1$.

\begin{propo} Let $A$ be a nilpotent algebra and let $B$ be a subalgebra of $A$. Then $B\subset I_A(B)$; in particular, all maximal subalgebras of $A$ are ideals of $A$.
\end{propo}
\begin{proof} Since $A$ is nilpotent, there is a $k$ such that $A^{[k]}\not \subseteq B$, but $A^{[k+1]}\subseteq B$. Then $B\subset A^{[k]}+B\subseteq I_A(B)$.
\end{proof}

\begin{defn} The {\bf Frattini subalgebra}, $F(A)$, of $A$ is the intersection of the maximal subalgebras of $A$; the {\bf Frattini ideal}, $\phi(A)$, is the largest ideal of $A$ contained in $F(A)$.
\end{defn}

\begin{coro}(\cite[Theorem 6]{gen}) If $A$ is a nilpotent algebra, then $\phi(A)=F(A)=A^2$.
\end{coro}

\begin{defn} If $B,C$ are ideals of an algebra $A$ with $B\subset C$, a {\bf chief series} of $A$ from $B$ to $C$ is a series $B=B_0\subset B_1\subset \ldots \subset B_r=C$, where the $B_i$ are ideals of $A$ and $B_{i+1}/B_i$ is a minimal ideal of $A/B_i$ for $0\leq i\leq r-1$. The factor algebras  $B_{i+1}/B_i$ in this series are called {\bf chief factors}.
\end{defn}

We also have the following characterisation of nilpotent algebras.

\begin{lemma}\label{gen} For an algebra $A$, the following are equivalent:
\begin{itemize}
\item[(i)] $A$ is nilpotent;
\item[(ii)] there is a chain of ideals of $A$
\[ 0=A_{(0)} \subset A_{(1)}\subset \ldots \subset A_{(n)}=A,
\] where $\dim A_{(i)}= i$ and $AA_{(i)}+A_{(i)}A\subseteq A_{(i-1)}$.
\end{itemize}
\end{lemma}
\begin{proof} This is part of \cite[Theorem 3]{gen}.
\end{proof}
\medskip

In general, weakly nilpotent algebras need not be nilpotent, as is shown in \cite{gen}. However, we have the following results.

\begin{lemma}\label{chief}  Let $A$ be an algebra in which $IJ$ is an ideal of $A$ for every pair of ideals $I,J\subseteq A$. Let $N$ be a right (respectively, left) nilpotent ideal of $A$ and let $B/C$ be a chief factor of $A$. Then $BN\subseteq C$ (respectively $NB\subseteq C$).
\end{lemma}
\begin{proof} Then $BN+C$ is an ideal of $A$ and $C\subseteq BN+C\subseteq B$. Hence $BN\subseteq C$ or $BN+C=B$. Suppose the latter holds. Then
\[ B=BN+C\subseteq (BN)N+C\subseteq ((BN)N)N+C= \ldots =C,
\] since $N$ is right nilpotent. Hence $BN\subseteq C$. Similarly, if $N$ is left nilpotent then $NB\subseteq C$. 
\end{proof}

\begin{theor}\label{nilp} Let $A$ be a weakly nilpotent algebra in which $IJ$ is an ideal of $A$ for every pair of ideals $I,J\subseteq A$. Then $A$ is nilpotent.
\end{theor}
\begin{proof} Let $B/C$ be a chief factor of $A$. Then $BA+AB\subseteq C$, by Lemma \ref{chief}. It follows that $\dim B/C=1$.
\par

Now let
\[ 0=A_{(0)} \subset A_{(1)}\subset \ldots \subset A_{(n)}=A,
\] be a chief series for $A$. Then $A$ is nilpotent by Lemma \ref{gen}.
\end{proof}

\begin{defn} We will call a class of algebras $\mathfrak{X}$ {\bf natural} if it is factor algebra and subalgebra closed and, for every $A\in \mathfrak{X}$, $IJ$ is an ideal of $A$ whenever $I,J$ are ideals of $A$.
\end{defn}

\begin{coro}\label{sub} Let $A\in \mathfrak{X}$, where $\mathfrak{X}$ is a natural class of algebras. Then every weakly nilpotent subalgebra of $A$ is nilpotent.
\end{coro}

\begin{coro}\label{nil}  Let $A\in \mathfrak{X}$, where $\mathfrak{X}$ is a natural class of algebras. Then $A$ has a maximal right nilpotent/ left nilpotent/nilpotent ideal.
\end{coro}
\begin{proof} It suffices to show that, if $N_1$ and $N_2$ are two right nilpotent ideals of $A$, then so is $N_1+N_2$. Now, $N_1+N_2$ is certainly a solvable ideal of $A$. Let $0=B_0\subset B_1\subset \ldots \subset B_n=N_1+N_2$ be a chief series for $N_1+N_2$. Then $B_n^2=B_n(N_1+N_2)=B_nN_1+B_nN_2\subseteq B_{n-1}$, by Lemma \ref{chief}. A straightforward induction proof then shows that $B_n^{n+1}=0$. It follows that $N_1+N_2$ is right nilpotent. Left nilpotency follows similarly. If $N_1,N_2$ are nilpotent, then $N_1+N_2$ is weakly nilpotent, and so nilpotent, by Theorem \ref{nilp}.
\end{proof}

\begin{defn} We will call the maximal right nilpotent (respectively, left nilpotent, nilpotent) ideal given in Corollary \ref{nil}, the {\bf right nilradical} (respectively, {\bf left nilradical}, {\bf nilradical}) of $A$, and denote it by $N_r(A)$ (respectively, $N_\ell(A)$, $N(A)$). Then $N(A)\subseteq N_r(A), N_{\ell}(A)\subseteq R(A)$, by \cite[Proposition 3.1]{ckm}.
\end{defn}

\section{Bicommutative algebras}
\begin{propo}\label{ideal} If $A$ is a bicommutative algebra and $I,J$ are ideals of $A$, then $IJ$ is an ideal of $A$.
\end{propo}
\begin{proof} Simply note that $A(IJ)\subseteq I(AJ)\subseteq IJ$ and $(IJ)A\subseteq (IA)J\subseteq IJ$.
\end{proof}
\medskip

Note that it follows from the above Corollary that the class of bicommutative algebras is natural and so has a nilradical. 

\medskip

We will make use of the following result of  Dzhumadil'daev and Tulenbaev which has some useful consequences..

\begin{theor}\label{dt1} (\cite[Theorem 1]{dt1}) Let $A$ be a bicommutative algebra. Then $A^2$ is commutative and associative.
\end{theor}

\begin{coro}\label{c1} If $A$ is a solvable bicommutative algebra, $A^2$ is nilpotent.
\end{coro}

\begin{coro} If $A$ is a bicommutative algebra which is either right or left nilpotent, then $A^2$ is nilpotent, and so $A$ is solvable.
\end{coro}

\begin{coro}\label{ar} Let $A$ be a bicommutative algebra and let R be its solvable radical. Then $RA$ and $AR$ are nilpotent ideals of $A$. 
\end{coro}

\begin{defn} Let $A$ be an algebra and let $a\in A$. Define $\rho_a : A \rightarrow A : x\mapsto xa$. Put $E_A^r(a)=\{x\in A \mid \rho_a^n(x)=0$ for some $n\}$.
Similarly, define $\lambda_a : A \rightarrow A : x\mapsto ax$. Put $E_A^l(a)=\{x\in A \mid \lambda_a^n(x)=0$ for some $n\}$.
\end{defn}

\begin{lemma} If $A$ is a right (respectively, left) commutative algebra and $a\in A$, then $E_A^r(a)$ (respectively, $E_A^l(a)$) is a subalgebra of $A$..
\end{lemma}
\begin{proof} Let $x,y\in E_A^r(a)$. Then there are $m,n$ such that $\rho_a^m(x)=\rho_a^n(y)=0$. Without loss of generality we can assume that $m=n$. But now $\rho_a^n(xy)=(\rho_a^n(x))y=0$, by repeated use of right commutativity, so $xy\in E_A^r(a)$. 
\par

The left version is similar.
\end{proof}

\begin{defn} If $A$ is an algebra, then a subset $B$ is called a {\bf left subideal} (respectively, {\bf right subideal}) of $A$ if there is a chain of subalgebras  $B=B_0\subseteq B_1\subseteq \ldots\subseteq B_r=A$ of $A$ with $B_i$ a left (respectively, right) ideal of $B_{i+1}$ for $0\leq i\leq r-1$.  Let $a\in A$. We call $a$ {\bf right nil} (respectively, {\bf left nil}) if $\rho_a^n(a)=0$ (respectively, $\lambda_a^n(a)=0$) for some $n\geq 0$. We will say that an ideal $B$ of $A$ is right nil (respectively, left nil) if every element of $B$ is right nil (respectively, left nil).
\end{defn}

\begin{theor}\label{factor} Let $B$ be a left (respectively, right) subideal of a right (respectively, left) commutative algebra $A$, and let $C$ be an ideal of $B$ with $C\subseteq \phi(A)$. If $B/C$ is right (respectively, left) nilpotent, then each element of $B$ acts right (respectively, left) nilpotently on $A$; in particular, $B$ is right (respectively, left) nil.
\end{theor}
\begin{proof} Let $b\in B$ and $B=B_0\subseteq B_1\subseteq \ldots\subseteq B_r=A$ be a chain of subalgebras of $A$ with $B_i$ a left ideal of $B_{i+1}$ for $0\leq i\leq r-1$. Then $\rho_b^r(A) \subseteq B$. Since $B/C$ is right nilpotent, there exists $s$ such that $\rho_b^s(B)\subseteq C$. Hence $\rho_b^{r+s}(A)\subseteq \phi(A)$. But $\rho_b^{r+s}(A)+E_A^r(b)=A$, by Fitting's Lemma, so $E_A^r(b)=A$ and each element of $B$ acts right  nilpotently on $A$. 
\par

Again, the left version is similar.
\end{proof}

\begin{coro}\label{phinil} If $A$ is a right (respectively, left) commutative algebra, then $\phi(A)$ is right (respectively, left) nil.
\end{coro}

\begin{coro}\label{biphi} If $A$ is a bicommutative algebra, then $\phi(A)$ is nilpotent.
\end{coro}
\begin{proof} This follows from Corollary \ref{phinil}, Theorem \ref{dt1} and the fact that $\phi(A)\subseteq A^2$, by \cite[Lemma 2.5]{frat}.
\end{proof}
\medskip

The following result concerning minimal ideals will be useful.

\begin{theor}\label{min1} Let $B$ be a minimal ideal in a bicommutative algebra $A$. If $R$ is the solvable radical of $A$, then, either $RB=0$ and $BR^2=0$, or $BR=0$ and $R^2B=0$.
\end{theor}
\begin{proof} We have that $RB$ is an ideal of $A$, by Proposition \ref{ideal}, so $RB=0$ or $RB=B$. Similarly, $BR=0$ or $BR=B$. Suppose that $RB=BR=B$. Then
\[ B=RB=R(BR)\subseteq BR^2 \hbox{ and } B=BR=(RB)R\subseteq R^2B.
\]
Suppose that $B\subseteq BR^{(k)}$ and $B\subseteq R^{(k)}B$. Then
\begin{align*}
B & \subseteq R^{(k)}B\subseteq R^{(k)}(BR^{(k)})\subseteq BR^{(k+1)} \hbox{ and }\\
 B & \subseteq BR^{(k)}\subseteq (R^{(k)}B)R^{(k)}\subseteq R^{(k+1)}B.
\end{align*}
Since $R$ is solvable and $B\neq 0$, either $RB=0$ or $BR=0$. If $RB=0$, then $BR^2\subseteq R(BR)=0$. Similarly, $R^2B=0$ if $BR=0$.
\end{proof}

\begin{propo}\label{bimax} Let $A$ be a  right  (respectively, left) nilpotent bicommutative algebra. Then every maximal subalgebra of $A$ is a left (respectively, right) ideal of $A$ and $A^3\subseteq \phi(A)$ (respectively, $^3A\subseteq \phi(A)$).
\end{propo}
\begin{proof} Let $A$ be a right nilpotent bicommutative algebra and let $M$ be a maximal subalgebra of $A$. Then there exists $k$ such that $A^k\not\subseteq M$ but $A^{k+1}\subseteq M$ and $A=M+A^k$. Hence $AM=M^2+A^kM\subseteq M$ and $M$ is a left ideal of $A$. Also, $A^3=A^2A\subseteq A^2M+A^2A^k\subseteq M+ (AA^k)A\subseteq M+A^{k+1}\subseteq M$.
\par

The result for a left nilpotent bicommutative algebra follows similarly.
\end{proof}

\begin{propo}\label{biann} Let $A$ be a bicommutative algebra and let $B$ be a subalgebra of $A$. Then $I_A(B)$ and $Ann_A(B)$ are subalgebras of $A$. If $B$ is an ideal of $A$ then so is $Ann_A(B)$
\end{propo}
\begin{proof} Let $B$ be a subalgebra of $A$ and let $a_1,a_2\in I_A(B)$. Then, for all $b\in B$, $(a_1a_2)b=(a_1b)a_2\in B$ and $b(a_1a_2)=a_1(ba_2)\in B$. The argument for $Ann_A(B)$ is similar.
\par

So now suppose that $B$ is an ideal of $A$ and let $b\in B$, $k\in K=Ann_A(B)$, $a\in A$. Then $(ak)b=(ab)k\in BK=0$, $(ka)b=(kb)a=0$, $b(ak)=a(bk)=0$ and $b(ka)=k(ba)\in KB=0$, whence $Ann_A(B)$ is an ideal of $A$.
\end{proof}
\medskip

The following property of minimal ideals of bicommutative algebras is shared by Novikov algebras.

\begin{theor}\label{simple} If $B$ is a minimal ideal of a bicommutative algebra $A$, then either $B^2=0$ or $B$ is a simple algebra.
\end{theor}
\begin{proof} Assume that $B^2\neq 0$, so $B^2=B$. Then $B\not \subseteq Ann_A(B)$, and so $B\cap Ann_A(B)=0$ by Proposition \ref{biann}. Let $C$ be a proper ideal of $B$. Then 
\[
BC= B^2C\subseteq (BC)B\subseteq CB \hbox{ and } CB=CB^2\subseteq B(CB)\subseteq BC,
\] so $BC=CB$. (Alternatively, simply note that $B,C\subseteq A^2$ which is commutative.) But now
\[
(BC)A\subseteq (BA)C\subseteq BC \hbox{ and } A(BC)=A(CB)\subseteq C(AB)\subseteq CB=BC,
\] so $BC$ is an ideal of $A$. Since $BC\subseteq C\neq B$ we have that $BC=0$. Also $CB=BC=0$, whence $C\subseteq B\cap Ann_A(B)=0$.
\end{proof}
\medskip

We can now give the follwing result, which leads to a characterisation of semisimple bicommutative algebras.

\begin{theor}\label{ss} Let $S$ be a finite-dimensional semisimple bicommutative algebra over any field $F$. Then every ideal of $S$ inside $S^2$ is a direct sum of simple minimal ideals of $S$.
\end{theor}
\begin{proof} Let $B$ be an ideal of $S$ inside $S^2$. We use induction on the length of a chief series of $S$ from $0$ to $B$. If $B$ is a minimal ideal of $S$ then it is simple, by Theorem \ref{simple}. So suppose the result holds when the length of such a chief series is less than $r+1$, and let the length from $0$ to $B$ be $r+1$. Let $C$ be an ideal of $S$ such that $B/C$ is a chief factor of $S$, so $C=S_1\oplus \ldots\oplus S_r$ where $S_i$ is a simple minimal ideal of $S$ for $1\leq i\leq r$.
\par

Let $R(B)$ be the radical of $B$. Then we have that $S_i(S_iR(B))=S_i^2R(B)=S_iR(B)$ and $(S_iR(B))S_i=S_i(R(B)S_i)\subseteq S_iR(B)$, since $S_i, R(B)\subseteq S^2$, which is associative. Thus, $S_iR(B)$ is an ideal of $S_i$. But $S_i$ is simple, so $S_iR(B)=S_i$ or $0$. The former is impossible, as a straightforward induction proof shows that $S_i=S_i(R(B)^{(k)})$ for all $k\geq 1$, whence $S_i=0$. It follows that $S_iR(B)=R(B)S_i=0$, and hence that $R(B)\subseteq Ann_S(C)\cap B$, which is an ideal of $S$, by Proposition \ref{biann}. 
\par

Now, $Ann_S(C)\cap B\subseteq C$ or $B=C+ Ann_S(C)\cap B$. If the former holds, then $R(B)=0$, so $B$ is a semisimple associative algebra, whence $B=C\oplus S_{r+1}$. Moreover, $S_{r+1}\subseteq Ann_S(C)\cap B$, so $B\subseteq C$, a contradiction. Hence $B=C+ Ann_S(C)\cap B$. Let $\sum_{i=1}^r s_i\in C\cap Ann_S(C)$, where $s_i\in S_i$.  Then $s_iS_i=S_is_i=0$, so $Fs_i$ is an ideal of the simple algebra $S_i$, whence $s_i=0$ for each $1\leq i \leq r$. We conclude that $B=C\oplus Ann_S(C)\cap B$ and $Ann_S(C)\cap B$ is a minimal ideal of $S$ and so must be simple.
\end{proof}

\begin{coro}\label{biss} Let $S$ be a finite-dimensional semisimple bicommutative algebra over any field. Then $S=S^2\dot{+} U$ where $S^2=S_1\oplus \ldots \oplus S_n$, the $S_i$ are simple ideals of $S$, $U^2=0$ and either $S_iU=0$ or $US_i=0$ for each $1\leq i \leq n$..
\end{coro}
\begin{proof} The fact that $S^2=S_1\oplus \ldots \oplus S_n$ follows from Theorem \ref{ss} above. There is a subalgebra $U$ of $S$ such that $S=S^2+U$ and $U\cap S^2\subseteq  \phi(U)$, by \cite[Lemma 7.1]{frat}. But $\phi(U)=0$ as in the proof of Theorem \ref{ss}. Clearly, $U^2\subseteq U\cap S^2=0$. Also $US_i=S_i$ or $0$. If $US_i=S_i$, then $S_iU=(US_i)U\subseteq U^2S_i=0$.
\end{proof}
\medskip

In the above result, $S^2$ is a commutative associative algebra, and so the $S_i$ are extension fields of the base field $F$, but $S$ need not be equal to $S^2$ as the following easy example shows.

\begin{ex} Let $A$ be spanned by $x,y$ where $x^2=x$, $xy=x$ and all other products are zero. It is easy to check that $A$ is a semisimple bicommutative algebra that is neither commutative nor associative.
\end{ex} 

In \cite{sz} it is shown that, for Novikov algebras, the sum of two solvable subalgebras is solvable, and the sum of two zero subalgebras is metabelian. Neither of these results hold for bicommutative algebras, as the above example also shows: note that $A=F(x-y)+Fy$ is its decomposition as a sum of two zero subalgebras.

\section{Assosymmetric algebras}
In this section, let $F$ denote a field of characteristic different from $2, 3$.  Kleinfeld proved the following in \cite{klein}.

\begin{theor}\label{k} Let $R$ be an asosymmetric ring of characteristic different from $2$ and $3$ which has no ideals $I\neq 0$ such that $I^2=0$, then $R$ is associative.
\end{theor}

\begin{coro} Every semisimple assosymmetric algebra over $F$ is associative.
\end{coro}

\begin{lemma}\label{prod} Let $I,J$ be ideals of the assosymmetric algebra $A$ over any field. Then $IJ$ is an ideal of $A$.
\end{lemma}
\begin{proof} We have
\begin{align*}
A(IJ)\subseteq & (AI)J+I(AJ)+(IA)J\subseteq IJ \hbox{ by left symmetry, and} \\
(IJ)A\subseteq & I(JA)+(IA)J+I(AJ)\subseteq IJ \hbox{ by right symmetry}.
\end{align*}
\end{proof}

Then Lemma \ref{prod} shows that the class of assosymmetric algebras is natural. It follows from Corollary \ref{nil} that every assosymmetric algebra has a maximal nilpotent ideal $N(A)$. Now, the following is proved in \cite{pr}.

\begin{theor} Let $A$ be a solvable assosymmetric algebra over $F$. Then $A$ is nilpotent.
\end{theor}

It follows that, for such an algebra $A$, $N(A)=R(A)$, the maximal solvable ideal of $A$ and we have the following result as a Corollary to Theorem \ref{k}.

\begin{coro}\label{fac} If $A$ is an assosymmetric algebra over $F$, then $A/N(A)$ is associative.
\end{coro} 

\begin{propo}\label{phi} Let $A$ be an assosymmetric algebra over $F$. Then $\phi(A)$ is nilpotent.
\end{propo}
\begin{proof} We have that $(\phi(A)+N(A))/N(A)\subseteq \phi(A/N(A))$, by \cite[Proposition 4.3]{frat}. But $\phi(A/N(A))\subseteq N(A/N(A))$, by \cite[Corollary 6.3]{frat}, so $\phi(A)\subseteq N(A)$, whence the result.
\end{proof}

\section{Novikov algebras}
The following three results were proved in \cite[Lemma 2.1, Theorem 3.3 and Corollary 3.6]{sz}.

\begin{lemma}\label{nproduct} If $I,J$ are ideals of a Novikov algebra $A$, then so is $IJ$. In particular, $A^n$, $^nA$, $A^{(n)}$ and $A^{[n]}$ are ideals of $A$ for all $n\geq 1$. Moreover, the annihilator of $I$ in $A$ is also an ideal of $A$.
\end{lemma}
\medskip

It follows from the above Lemma that the class of Novikov algebras is natural.

\begin{theor}\label{nov} Let $A$ be a Novikov algebra. Then the following are equivalent:
\begin{itemize}
\item[(i)] $A$ is right nilpotent;
\item[(ii)] $A^2$ is nilpotent; 
\item[(iii)] $A$ is solvable.
\end{itemize}
\end{theor}

\begin{coro}\label{left} Every left nilpotent Novikov algebra is nilpotent.
\end{coro}

\begin{coro}\label{novphi} If $A$ is a solvable Novikov algebra, then $\phi(A)$ is nilpotent.
\end{coro}
\begin{proof} This follows from the fact that $\phi(A)\subseteq A^2$.
\end{proof}

\begin{propo}\label{novar} Let $A$ be a Novikov algebra with solvable radical $R$. Then $AR$ is a nilpotent ideal of $A$.
\end{propo}
\begin{proof}  We use induction on the maximum length $k$ of a chief series of $A$ from $0$ to $R$. If $R$ is a minimal ideal of $A$, then $R^2=0$, so $(AR)^2\subseteq R^2=0$ and $^2(AR)\subseteq R^2=0$, so the result holds for $k=1$. So, suppose it holds whenever $k\leq n$ ($n\geq 1$), and let $A$ be such that $k=n+1$.
Let $B$ be a minmal ideal of $A$ with $B\subseteq R$. Then $R/B$ is the solvable radical of $A/B$ and so $\frac{A}{B}\frac{R}{B}$ is nilpotent. Hence, there is an $r$ such that $^r(AR)\subseteq B$. But now $BR=0$, since $R$ is right nilpotent, and so
\[
^{r+1}(AR)\subseteq (AR)B\subseteq (AB)R\subseteq BR=0.
\] Thus, $AR$ is left nilpotent, and so is nilpotent, by Corollary \ref{left}. The result follows by induction.
\end{proof}

\begin{propo}\label{ann} Let $A$ be a Novikov algebra and let $B$ be a subalgebra of $A$. Then $I_A(B)$ and $Ann_A(B)$ are subalgebras of $A$. If $B$ is an ideal of $A$ then so is $Ann_A(B)$
\end{propo}
\begin{proof} Let $B$ be a subalgebra of $A$ and let $a_1,a_2\in I_A(B)$. Then, for all $b\in B$, $(a_1a_2)b=(a_1b)a_2\in B$ and $b(a_1a_2)=(ba_1)a_2+a_1(ba_2)-(a_1b)a_2\in B$. The argument for $Ann_A(B)$ is similar.
\par

So now suppose that $B$ is an ideal of $A$. Then $Ann_A(B)$ is an ideal of $A$ is shown in \cite[Lemma 2.1]{sz}.
\end{proof}
\medskip

The following was proved by Zelmanov in \cite[Proposition 2]{zel}. 

\begin{propo}\label{novss} Let $S$ be a semisimple Novikov algebra over a field of characteristic zero. Then $S$ is a direct sum of fields.
\end{propo}

\section{$\phi$-free algebras}
\begin{defn} An ideal $B$ of an algebra $A$ is called a {\bf zero ideal} if $B^2=0$. The {\bf socle}, denoted $Soc(A)$ (respectively, {\bf zero socle}, denoted $Zsoc(A)$) of $A$ is the sum of the minimal ideals (respectively, minimal zero ideals) of $A$. We will say that $A$ is {\bf $\phi$-free} if $\phi(A)=0$.
\end{defn}

\begin{theor}\label{split}  Let $A$ be an algebra in which $\phi(A)$ is nilpotent. Then $A$ is $\phi$-free if and only if it splits over its zero socle.
\end{theor}
\begin{proof} If $A$ is $\phi$-free, it splits over its zero socle, by  \cite[Lemma 7.2]{frat}. So suppose that $A=Zsoc(A)\dot{+} C$, where $C$ is a subalgebra of $A$ and $Zsoc(A)=Z_1\oplus\ldots\oplus Z_n$, where $Z_i$ is a minimal zero ideal of $A$, and that $\phi(A)\neq 0$. Then there is a minimal ideal $Z$ of $A$ contained in $\phi(A)$. Since $\phi(A)$ is nilpotent, we have that $Z$ is a zero ideal. Thus $\phi(A)\cap Zsoc(A)\neq 0$. Put $M_i=(Z_1\oplus\ldots\oplus \hat{Z_i}\oplus\ldots\oplus Z_n)\dot{+}C$, where $\hat{Z_i}$ indicates a term that is missing, and let $K_i$ be a subalgebra of $A$ containing $M_i$. Then $A=M_i+Z_i=K_i+Z_i$ and $K_i=M_i+K_i\cap Z_i$. But $A(K_i\cap Z_i)=(K_i+Z_i)(K_i\cap Z_i)\subseteq K_i\cap Z_i$, since $Z_i$ is a zero algebra. Similarly, $(K_i\cap Z_i)A\subseteq K_i\cap Z_i$, so $K_i\cap Z_i$ is an ideal of $A$ inside $Z_i$. It follows that $K_i\cap Z_i=Z_i$ or $0$, whence $K_i=A$ or $M_i$. Thus, $M_i$ is a maximal subalgebra of $A$ for each $1\leq i\leq n$. Hence
\[ \phi(A)\subseteq \cap_{i=1}^n M_i \subseteq C,
\] and $\phi(A)\cap Zsoc(A) =0$, a contradiction. It follows that $\phi(A)=0$.
\end{proof}

\begin{coro}\label{csplit} Let $A$ be a bicommutative, assosymmetric over a field of characteristic different from $2, 3$ or solvable Novikov algebra. Then $A$ is $\phi$-free if and only if it splits over its zero socle.
\end{coro}
\begin{proof} This follows from Theorem \ref{split}, Proposition \ref{phi} and Corollaries \ref{biphi} and \ref{novphi}.
\end{proof}

\begin{theor}\label{t} Let $A$ be a $\phi$-free bicommutative, assosymmetric or Novikov algebra. Then $Zsoc(A)=N(A)=Ann_A(Soc(A))$.
\end{theor}
\begin{proof} Clearly, $Zsoc(A)\subseteq N(A)$. Let $B$ be a minimal ideal of $A$ and let $N$ be a nilpotent ideal of $A$. Then $B\cap N=0$ or $B$. If the former holds, then $BN+NB\subseteq B\cap N=0$, so $N\subseteq Ann_A(B)$. So suppose that the latter holds. Then $B\subseteq N$ and so $B\subseteq Ann_A(N)$, by Lemma \ref{chief}, whence $BN+NB=0$ again. It follows that $N(A)\subseteq Ann_A(Soc(A))$. It now suffices to show that $Ann_A(Soc(A))\subseteq Zsoc(A)$.
\par

There is a subalgebra $C$ of $A$ such that $A=Zsoc(A)\dot{+}C$, by \cite[Lemma 7.2]{frat}. Now, $Ann_A(Soc(A))\cap C$ is an ideal of $A$ and so must contain a minimal ideal $D$. But $D$ must be a zero ideal, since it annihilates itself, by assumption. Hence $D\subseteq Zsoc(A)\cap C=0$ and the result is proved.
\end{proof}
\medskip

In the next results we seek to characterise the $\phi$-free bicommutative and Novikov algebras.

\begin{theor}\label{biphifree} Let $A$ be a bicommutative algebra with solvable radical $R$. Then $A$ is $\phi$-free if and only if $A=Zsoc(A)\dot{+}(D\oplus E)$, where $D$ is a zero subalgebra of $A$, $R=Zsoc(A)\dot{+}D$, $E=E^2\dot{+} U$, $E^2=S_1\oplus \ldots \oplus S_n$, $S_i$ is a simple commutative associative ideal of $E$, $U^2=0$, either $S_iU=0$ or $US_i=0$ for $1\leq i\leq n$, and $Zsoc(A)=Z_1\oplus Z_2$ where $Z_1R=0$, $RZ_2=0$.
\end{theor}
\begin{proof} Suppose first that $A$ is $\phi$-free. We have that $A=Zsoc(A)\dot{+}C$, by Corollary \ref{csplit}. Now $(C\cap R)C\subseteq RA\cap C\subseteq N(A)\cap C=0$ and $C(C\cap R)\subseteq AR\cap C\subseteq N(A)\cap C=0$, by Corollary \ref{ar}. Put $C\cap R=D$. If $R(C)$ is the solvable radical of $C$, then $Zsoc(A)+R(C)$ is a solvable ideal of $A$, and so $R(C)=R\cap C=D$ and $D^2=0$.
\par

Clearly, $\phi(C)\subseteq R(C)=D$, by Corollary \ref{biphi}, so $C\phi(C)=\phi(C)C=0$. Moreover, for every minimal zero ideal $Z$, 
\begin{align*}
(Z\phi(C))A&=(Z\phi(C))C\subseteq (ZC)\phi(C)\subseteq Z\phi(C) \hbox{ and} \\
A(Z\phi(C))&=C(Z\phi(C))\subseteq Z(C\phi(C))\subseteq Z\phi(C).
\end{align*} Thus, $Z\phi(C)$ is an ideal of $A$, whence  $Z\phi(C)=0$ or $Z\phi(C)=Z$. Suppose the latter holds. Then $ZR\neq 0$, so $RZ=0$, by Theorem \ref{min1}. It follows that $\phi(C)Z=0$. But $\phi(C)$ and $Z$ are in $A^2$, which is commutative, so we have a contradiction. Hence, $Z\phi(C)=0$. Similarly, $\phi(C)Z=0$, and so $\phi(C)$ is an ideal of $A$. This implies that $\phi(C)\subseteq \phi(A)=0$, by \cite[Lemma 4.1]{frat}.
\par

Since $D$ is a zero ideal of $C$, $C= D\dot{+} E$ where $E$ is a semisimple subalgebra of $C$, by \cite[Lemma 7.2]{frat}. Moreover, $DE+ED\subseteq (RA\cap C)+(AR\cap C)=0$, as in the first paragraph of the proof, so $C=D\oplus E$. Hence $E=E^2\dot{+} U$ where $E^2=S_1\oplus \ldots \oplus S_n$, the $S_i$ are simple ideals of $E$, $U^2=0$ and either $S_iU=0$ or $US_i=0$ for each $1\leq i \leq n$, by Corollary \ref{biss}.
\par

Let $Z$ be a minimal ideal in $Zsoc(A)$. Then $ZR=0$ or $ZR=Z$. Suppose the latter holds. Then
\[ RZ=R(ZR)\subseteq ZR^2=0,
\] by Theorem \ref{min1}. This completes the proof of necessity.
\par

The converse follows from Corollary \ref{csplit}.
\end{proof}

\begin{theor}\label{phifree}  Let $A$ be a $\phi$-free Novikov algebra with solvable radical $R$. Then $A=Zsoc(A)\dot{+}C$, where $C$ is a subalgebra of $A$ and $A(C\cap R)=0$.
\end{theor}
\begin{proof} We have that $A=Zsoc(A)\dot{+}C$, by Corollary \ref{csplit}. Also $C(C\cap R)\subseteq AR\cap C\subseteq Zsoc(A)\cap C=0$, by Proposition \ref{novar}. Moreover, $Zsoc(A)(C\cap R)\subseteq Zsoc(A)R$. Now, if $Z$ is a minimal zero ideal of $A$, $ZR=Z$ or $0$. But $ZR=Z$ implies that $Z=ZR=(ZR)R= \ldots =0$  since $R$ is right nilpotent and $Z\subseteq R$. Hence $A(C\cap R)=0$.
\end{proof}

\begin{coro}\label{arr} Let $A$ be a Novikov algebra with solvable radical $R$. Then $(AR)R\subseteq \phi(A)\subseteq A^2$.
\end{coro}
\begin{proof} From Theorem \ref{phifree}, we have that, if $A$ is $\phi$-free, then $AR\subseteq Zsoc(A)$, so $(AR)R=0$. Now $\phi(A/\phi(A))=0$ (see \cite[section 7]{frat}). Also, if $R(A/\phi(A))=\Gamma/\phi(A)$, then $\Gamma$ is solvable since $\phi(A)$ is solvable, so $R(A/\phi(A))=R/\phi(A)$,
\end{proof}

\begin{theor} Let $A$ be a Novikov algebra with solvable radical $R$ over a field $F$ of characteristic zero. If $R$ is nilpotent, then $A$ is $\phi$-free if and only if $A=R\dot{+} S$ where $R$ is a zero algebra and $S$ is a semisimple commutative associative algebra; that is, $S=\oplus_{i=1}^n S_i$, where $S_i$ is a field for $1\leq i\leq n$.
\end{theor}
\begin{proof} Let $A$ be $\phi$-free. Then $A=R\dot{+} S$, where $R$ is a zero algebra and $S$ is a semisimple subalgebra of $A$. Also, $S$ has the claimed form, by Proposition \ref{novss}.
\par

The converse follows from Corollary \ref{csplit}, since $R$ is an $S$-module and so is completely reducible.
\end{proof}

\begin{coro}\label{rad}  Let $A$ be a Novikov algebra with solvable radical $R$ over a field $F$ of characteristic zero. If $R$ is nilpotent, then $A$ is $\phi$-free if and only if its radical is a zero algebra.
\end{coro}

\begin{coro}\label{rad1}  Let $A$ be a Novikov algebra with solvable radical $R$ over a field $F$ of characteristic zero. Then $\phi(A)\subseteq R^2$ and so $\phi(A)$ is nilpotent..
\end{coro}
\begin{proof} The radical of $A/R^2$ is a zero algebra, and so $\phi(A/R^2)=0$, by Corollary \ref{rad}. It follows from \cite[Corollary 4.4]{frat} that $\phi(A)\subseteq R^2$. Also, $\phi(A)$ is nilpotent, by Theorem \ref{nov}.
\end{proof}

\begin{coro}  Let $A$ be a Novikov algebra with solvable radical $R$ over a field $F$ of characteristic zero. If $R$ is nilpotent, then $\phi(A)=R^2$.
\end{coro}
\begin{proof} We have that $\phi(A)\subseteq R^2$ from Corollary \ref{rad1}. But $R^2=\phi(R)\subseteq \phi(A)$, by \cite[Theorem 6]{gen} and \cite[Lemma 4.1]{frat}.
\end{proof}

\begin{lemma}\label{a3} If $A$ be an algebra with $A^3=0$, then $A$ is a Novikov algebra if and only if it is bicommutative.
\end{lemma}
\begin{proof} Substituting $A^3=0$ into left symmetry yields left commutativity.
\end{proof}
\medskip

Finally, we have two results concerning maximal subalgebras.

\begin{theor}\label{novmax} Let $A$ be a Novikov algebra  over a field  $F$, and consider the following statements. 
\begin{itemize}
\item[(i)] $A$ is right nilpotent; 
\item[(ii)] $A^3\subseteq \phi(A)$; 
\item[(iii)] all maximal subalgebras of $A$ are left ideals of $A$.
\end{itemize}
Then $(i)\Rightarrow (ii)$, $(ii)\Rightarrow (iii)$.  
\end{theor}
\begin{proof} $(i)\Rightarrow (ii)$: This follows from Theorem \ref{nov} and Corollary \ref{arr}.
\par
\noindent $(ii)\Rightarrow (iii)$: This follows from Lemma \ref{a3} and Proposition \ref{bimax}.
\end{proof}

\begin{theor} Let $A$ be a solvable bicommutative algebra. Then the following are equivalent.
\begin{itemize}
\item[(i)] $A^3\subseteq \phi(A)$;
\item[(ii)] all maximal subalgebras of $A$ are left ideals of $A$.
\end{itemize}
\end{theor}
\begin{proof} $(i)\Rightarrow (ii)$: This follows from Lemma \ref{a3} and Theorem \ref{novmax}.
\par

\noindent $(ii)\Rightarrow (i)$: Suppose first that $A$ is $\phi$-free. Then $A=\oplus_{i=1}^nZ_i\dot{+} D$ where $Z_i$ is a minimal zero ideal of $A$ and $D^2=0$, by Theorem \ref{biphifree}. Suppose that $Z_kD=Z_k$ for some $1\leq k\leq n$. Then $M_k=Z_1\oplus \ldots \oplus \hat{Z_k}\oplus \ldots \oplus Z_n+D$, where the hat indicates a term that is missing from the sum, is a maximal subalgebra of $A$. Since it is a left ideal of $A$, we have that $Z_kM_k=Z_kD\subseteq Z_k\cap M_k=0$. Hence $A^3=A^2A=(\oplus_{i=1}^nZ_i)A=0$. The result now follows by considering $A/\phi(A)$.
\end{proof}
\medskip

{\bf Acknowledgement}: The author is grateful to the referee for their careful reading of the paper and for their helpful, detailed suggestions.

\end{document}